\documentclass[leqno,11pt,a4paper]{amsart}
\usepackage[margin=2.5cm]{geometry}

\usepackage[usenames,dvipsnames]{color}
\usepackage{hyperref}
\usepackage{pdflscape}
\usepackage{longtable}
\usepackage{booktabs}
\usepackage{colortbl}
\usepackage{arydshln}
\usepackage{verbatim}
\usepackage{enumitem}
\newtheorem{thm}{Theorem}[section]

\newtheorem{cor}[thm]{Corollary}
\newtheorem{prop}[thm]{Proposition}
\theoremstyle{definition}
\newtheorem{eg}[thm]{Example}

\newtheorem{defn}[thm]{Definition}

\numberwithin{equation}{section}
\newcommand{\C}{\mathbb{C}}
\newcommand{\PP}{\mathbb{P}}
\newcommand{\F}{\mathbb{F}}
\newcommand{\Q}{\mathbb{Q}}
\newcommand{\Z}{\mathbb{Z}}

\newcommand{\cO}{\mathcal{O}}
\newcommand{\cB}{\mathcal{B}}

\newcommand{\Pnum}{P_{\mathrm{num}}}
\newcommand{\Porb}{P_{\mathrm{orb}}}
\newcommand{\Pini}{P_{\mathrm{ini}}}

\newcommand{\abs}[1]{\left\vert{#1}\right\vert}
\newcommand{\modb}[1]{\left(\mathrm{mod}\ {#1}\right)}
\newcommand{\qfact}{\hbox{$\Q$-factorial}}

\newcommand{\tV}{{\widetilde{V}}}
\renewcommand{\phi}{\varphi}

\DeclareMathOperator{\wt}{wt}
\DeclareMathOperator{\Gr}{Gr}
\DeclareMathOperator{\CGr}{CGr}
\DeclareMathOperator{\OGr}{OGr}

\DeclareMathOperator{\Pf}{Pf}
\DeclareMathOperator{\Proj}{Proj}
\DeclareMathOperator{\Spec}{Spec}
\newcommand{\GIT}{/\!\!/}
\newcommand{\evnrow}{\rowcolor[gray]{0.95}}
\newcommand{\oddrow}{}
\begin{document}
\author[G.~Brown]{Gavin Brown}
\address{School of Mathematics\\Loughborough University\\Loughborough\\LE$11$\ $3$TU\\UK}
\email{G.D.Brown@lboro.ac.uk}
\author[A.~M.~Kasprzyk]{Alexander Kasprzyk}
\address{Department of Mathematics\\Imperial College London\\London\\SW$7$\ $2$AZ\\UK}
\email{a.m.kasprzyk@imperial.ac.uk}
\author[L.~Zhu]{Lei Zhu}
\address{Fudan University\\Shanghai\\China}
\email{051018003@fudan.edu.cn}
\title{Gorenstein formats, canonical and Calabi--Yau threefolds}
\begin{abstract}
We extend the known classification of threefolds of general type that are complete intersections to various classes of non-complete intersections, and find other classes of polarised varieties, including Calabi--Yau threefolds with canonical singularities, that are not complete intersections. Our methods apply more generally to construct orbifolds described by equations in given Gorenstein formats.
\end{abstract}
\maketitle
\section{The equations of canonical threefolds}
In this paper, a \emph{threefold} is a complex three-dimensional projective variety with \qfact\ canonical singularities, and a \emph{canonical threefold} is one that has ample canonical class. By the minimal model program, any threefold of general type is birational to a unique canonical threefold, its \emph{canonical model}, and so for birational classification it is enough to classify canonical threefolds. As outlined by Corti and Reid~\cite{CR00}, an explicit classification of varieties begins to do this for varieties that can be described by small sets of equations. Complete intersections in projective space provide many examples---see the original geographical considerations and map of Persson~\cite[\S2]{persson} or the Calabi--Yau map of Candelas, Lynker, and Schimmrigk~\cite{CLS}---but here we are interested in other cases.

Given a threefold $V$, its canonical model is $X=\Proj R(V,K_V)$ where the can\-onical ring $R(V,K_V)=\bigoplus_{m\ge 0} H^0(V,mK_V)$. For example, a nonsingular sextic hypersurface $X_6\subset\PP^4$ is a canonical threefold, and $R(X,K_X)$ is isomorphic to its homogeneous coordinate ring. Since the canonical ring is rarely generated in degree one, canonical threefolds often lie in weighted projective space: the double cover of $\PP^3$ branched in a nonsingular surface of degree ten is a hypersurface $X_{10}\subset\PP(1,1,1,1,5)$ and $R(X,K_X)$ is generated by these weighted homogeneous coordinates, graded in degrees $1, 1, 1, 1, 5$. Iano-Fletcher~\cite[Table~3]{Fletcher} lists $23$ families of such weighted canonical hypersurfaces, the most exotic being $X_{46} \subset \PP(4,5,6,7,23)$, and Chen, Chen, and Chen~\cite{CCC} show that this is the complete list of canonical hypersurfaces.

In this paper we find new families of canonical threefolds that are not complete intersections, and we provide a systematic method for computing all cases up to a prescribed bound. In the notation of~\S\ref{sec!formats} we obtain the following result:

\begin{thm}\label{thm!main}
There are $18$ families of canonical threefolds whose general member embeds pluricanonically in six-dimensional weighted projective space as a codimension three subvariety with equations in weighted Grassmannian $\Gr(2,5)$ format. These families are described in Table~\ref{tab!codim3}.
\end{thm}

Our method is based on the orbifold Riemann--Roch formula of Buckley, Reid, and Zhou \cite{BRZ}, which we state in our context as Theorem~\ref{thm!icecream}. We show that the terminal singularities arising on canonical threefolds make strictly positive contributions to this formula (Theorem~\ref{thm!C_coeffs}), which bounds the number of possible baskets of singularities for given invariants.

It seems likely that the $18$ families of Theorem~\ref{thm!main} realise all canonical threefolds in codimension three, apart from complete intersections and their degenerations, but we do not prove this. The main point is that these are not complete intersections, and so do not appear in~\cite{Fletcher}. We can certainly go much further with these constructions using different formats: Corti and Reid~\cite{CR} find an example in codimension five, and we find $21$ families in the same codimension five orthogonal Grassmannian format (Table~\ref{tab!codim5}). We also find $57$ families in codimension four arising as hypersurface sections through four-folds in $\Gr(2,5)$ format.

A \emph{Calabi--Yau threefold} is a threefold with $K_X=0$ and $h^1(X,\cO_X)=h^2(X,\cO_X)=0$ and canonical singularities. We restrict to orbifolds having only isolated orbifold points of the form $\frac{1}r(a,b,c)$ with $a+b+c\equiv0\ \modb{r}$; these are the isolated three-dimensional cyclic quotient singularities that admit crepant resolutions, so each of our examples has a resolution to a Calabi--Yau manifold. We use the same methods to describe families of Calabi--Yau threefolds in various formats. In contrast to canonical threefolds, in this case the Riemann--Roch contributions of singularities need not be linearly independent; for example, the pair $\frac{1}3(1,1,1)$ and $\frac{1}3(2,2,2)$ make opposite contributions. This rarely causes confusion in the low-codimensional models we describe, but it does mean our purely numerical arguments can at first sight have infinitely many possible baskets of singularities to report.

Another contrast with canonical threefolds is that lists of Calabi--Yau threefolds tend to be large. It is certainly not the case that the examples we find exhaust all possible Calabi--Yau threefolds in the formats we consider. Nevertheless there has been a great deal of work to describe Calabi--Yau threefolds, and our examples extend some known lists already in the literature, such as the nonsingular examples of Tonolli~\cite{tonoli} and Bertin~\cite{bertin}. Table~\ref{tab!summary} summarises our results; detailed lists are available online at~\cite{grdb}.

More generally these methods can be used to find examples of polarised $d$-dimensional orbifolds $X,A$ with prescribed canonical class $K_X=kA$, for any integer~$k$, that have isolated orbifold locus. This restriction is imposed only because we do not know the contribution to orbifold Riemann--Roch of higher-dimensional orbifold strata; but see~\cite{selig} for recent progress. We have computer code, written for the computational algebra system Magma~\cite{magma} and available for download at~\cite{grdb}, that can make such searches systematically.

Goto and Watanabe~\cite{GW} describe graded rings of the type $R(X,D)$, computing their cohomology and characterising those that are Gorenstein, which includes all cases we consider:

\begin{thm}[\protect{\cite[5.1.9--11]{GW}}]
\label{thm!GW}
Let $X$ be a projective variety and $D$ an ample divisor. Set $R=R(X,D)$, the corresponding graded ring, so that $X=\Proj R$. If $R$ is Cohen--Macaulay then:
\begin{enumerate}
\item
$H^i(X,\cO_X)=0$ for $0<i<\dim X$;
\item
$R$ is Gorenstein if and only if $K_X = kD$ for some integer~$k$.
\end{enumerate}
\end{thm}

\section{Formats and candidate varieties}
\subsection{Regular pullbacks from key varieties}\label{sec!formats}
A format describes a presentation of the equations of a variety, for example by saying that the equations are minors of some matrix. Informal notions of format for polynomial equations appear regularly, sometimes describing a component of a Hilbert scheme or capturing some other feature of the geometry, and there are more formal prescriptions such as~\cite[\S12]{stevens}. We define format to suit our applications, loosely following Dicks and Reid~\cite[Theorem~3.3]{dicksreid},
~\cite[\S1.5]{fun}:

\begin{defn}
A \emph{Gorenstein format $F$ of codimension $c$} is a triple $(\tV,\chi,\F)$ consisting of:
\begin{enumerate}
\item
A Gorenstein (in particular, Cohen--Macaulay) affine variety $\tV\subset\C^n$ of codimension~$c$, which we refer to as the \emph{key variety} of the format;
\item
\label{cond:wt}
A diagonal $\C^*$ action on $\tV$ with strictly positive weights $\chi$, which we refer to as the \emph{key weights} of the format;
\item
\label{cond:res}
A graded minimal free resolution $\F$ of $\cO_{\tV}$ as a graded $\cO_{\C^n}$-module.
\end{enumerate}
\end{defn}

The $\C^*$ actions on $\C^n$ that are compatible with its toric structure are parametrised by the character lattice $N_{\C^n} =\Z^n$, and the positive actions are those lying strictly in the positive quadrant $Q\subset N_{\C^n}$. A subset $\Lambda\subset N_{\C^n}$ of these actions leave $\tV$ invariant, and condition~\eqref{cond:wt} asserts that $\Lambda\cap Q$ is not empty. We need a little more: that the given free resolution $\F$ is equivariant for the action. In many cases we consider the key variety has monomial syzygies, so the homogeneity of the equations of $\tV$ is enough, and $\Lambda\cap Q$ is some (infinite) polyhedron in $Q$. We then iterate over the formats by enumerating the points of $\Lambda\cap Q$.

Condition~\eqref{cond:res} determines the \emph{Hilbert numerator} $\Pnum(t)$ of the format: $\Pnum(t) = 1 - \sum t^{d_i} + \sum t^{e_j} - \cdots + (-1)^ct^k$, where $d_i$ are the degrees of the equations, $e_j$ the degrees of the first syzygies, and so on, and $k$ is the adjunction number of $\F$. This polynomial has Gorenstein symmetry: $t^k\Pnum(1/t) = (-1)^c \Pnum(t)$. It determines the Hilbert series, as in Proposition~\ref{prop!PX} below.

One could imagine other definitions of format, both weaker and stronger, but this one is well adapted to our applications.

Let $F=(\tV,\chi,\F)$ be a Gorenstein format of codimension~$c$. We construct Gorenstein varieties $X\subset \PP^{d+c}(W)$ of codimension~$c$ and dimension $d$ in weighted projective space, with weights $W$, as \emph{regular pullbacks}, which we recall from~\cite[\S1.5]{fun}:

\begin{prop}\label{prop!reg}
Let $(\tV\subset\C^n,\chi,\F)$ be a Gorenstein format of codimension~$c$. Let $R$ be a polynomial ring and $\phi\colon \Spec R \rightarrow \C^n$ a morphism. The following are equivalent:
\begin{enumerate}
\item
$\phi^{-1}(\tV)\subset\Spec R$ has codimension~$c$;
\item
The pullback of $\F$ by $\phi$ is a free resolution of $R$-modules;
\item
$x_i - \phi^*(x_i)$ for $i=1,\dots, n$ form a regular sequence on $\Spec R\times\C^n$, where $x_1,\dots,x_n$ are the coordinates of $\C^n$.
\end{enumerate}
If these conditions hold then $\phi^{-1}(\tV)\subset\Spec R$ is called a \emph{regular pullback of $\tV$}, and is a Gorenstein affine variety. Furthermore, if $\phi$ is graded of degree zero then the pullback of $\F$ by $\phi$ is a graded minimal free resolution of $R$-modules with the same Hilbert numerator as $\F$.
\end{prop}

Fix any integer $d>0$, the dimension of the varieties $X$ that we seek. Let $F=(\tV,\chi,\F)$ be a Gorenstein format of codimension~$c$ and fix a graded polynomial ring $R$ with $d+c+1$ variables and strictly positive weights~$W$. If $\phi\colon\Spec R \rightarrow\C^n$ is graded of degree zero and $\phi^{-1}(\tV)\subset \Spec R$ is a regular pullback containing the origin $O\in\Spec R$, then we define the \emph{projectivised regular pullback} to be
\[
X =
\phi^{-1}(\tV)\GIT_{\! W}\,\C^* =
\left(\phi^{-1}(\tV)
\setminus O\right)/\C^*\subset\PP(W).
\]
The next proposition follows immediately: the Hilbert series of $X$ is determined by the graded Betti numbers of a free resolution, and since $\phi$ satisfies the conditions of Proposition~\ref{prop!reg} and has degree zero, the graded Betti numbers are exactly those of $\F$ with grading~$\chi$.

\begin{prop}\label{prop!PX}
Let $F=(\tV\subset\C^n,\chi,\F)$ be a Gorenstein format of codimension~$c$, $R$ a polynomial ring graded by strictly positive weights $W$ with a morphism $\phi\colon\Spec R \rightarrow\C^n$ graded of degree zero. Then every projectivised regular pullback $X\subset\PP(W)$ has Hilbert series
\[
P_X(t) = \Pnum(t) \Big/ \prod_{a\in W} (1-t^a)
\]
where $\Pnum(t)$ is the Hilbert numerator of the format $F$.

If, in addition, $X$ is an irreducible variety that is well-formed as a subvariety of $\PP(W)$ then the canonical sheaf of $X$ is $\omega_X = \cO_X(k_{\tV}-\alpha)$, where $\alpha$ is the sum of the weights $W$ and $k_{\tV}=\deg\Pnum(t)$ is the adjunction number of $\F$.
\end{prop}

Recall that $X\subset\PP(W)$ is well formed if the intersection of $X$ with any non-trivial orbifold locus of $\PP(W)$ has codimension at least two in $X$; see~\cite[Definition~6.9]{Fletcher}.

\begin{defn}
A \emph{candidate variety} is a format $F = (\tV,\chi,\F)$ of codimension~$c$ together with a morphism $\phi\colon \Spec R \rightarrow \C^n$ of degree zero from a graded polynomial ring $R$ that satisfies the equivalent conditions of Proposition~\ref{prop!reg}. A candidate variety is \emph{well-formed} if the projectivised regular pullback $X\subset\PP(W)$ is well-formed as a subvariety.
\end{defn}

Approximately speaking, we think of a candidate variety as representing general members of a family of varieties in a common weighted projective space whose equations and syzygies are modelled on a common free resolution $\F$. The condition only asks for a single map, although in the practical situations we encounter below any sufficiently general map will work. The space of maps $\Spec R \rightarrow \C^n$ of degree zero that give regular pullbacks may have more than one component, but we do not consider this question at all.

\begin{eg}\label{sec!pfaff}
Following~\cite{CR}, let $\tV = \CGr(2,5)\subset\C^{10}$ be the affine cone over the Grassmannian $\Gr(2,5)$ in its Pl\"ucker embedding. The equations of $\tV$ are the maximal Pfaffians of a generic skew $5\times 5$ matrix
\[
M =\begin{pmatrix}
x_1&x_2&x_3&x_4\\
&x_5&x_6&x_7\\
&&x_8&x_9\\
&&&x_{10}
\end{pmatrix}
\]
(we write only the strict upper-triangular part of such matrices). These equations are homogeneous with respect to a five-parameter system of weights $\Z^5=\Lambda\subset\Z^{10}$, which one can determine by enforcing homogeneity of these Pfaffians.

We can use $\tV$ as a key variety to find K3 surfaces. Let $\chi = (3,4,4,5,5,5,6,6,7,7)\in\Lambda$, which we understand better in matrix form as
\[
\chi =\begin{pmatrix}
3&4&4&5\\
&5&5&6\\
&&6&7\\
&&&7
\end{pmatrix}.
\]
This has Hilbert numerator
\[
\Pnum = 1 - t^9 - 2t^{10} - t^{11} - t^{12} + t^{14} + t^{15} + 2t^{16} + t^{17} - t^{26}.
\]
Taking a suitable map of $\PP(a_0,\dots,a_5)$ with $a_0+\cdots+a_5=26$ may describe a family of K3 surfaces, since at least the canonical class is right and $h^1(X,\cO_X)=0$ by Theorem~\ref{thm!GW}. In this case, maps from either $\PP(1,3,4,5,6,7)$ or $\PP(2,3,4,5,5,7)$ work, and these are two families in Alt{\i}nok's list~\cite{A} of $69$ codimension three K3 surfaces in $\Gr(2,5)$ format.

The weighted projective space $\PP(1,3,4,5,6,7)$ also admits a map to a different $\Gr(2,5)$ format with grading $\chi=(1,3,4,5,4,5,6,7,8,9)\in\Lambda$, with $\Pnum = 1 - t^8 - t^9 - t^{10} - t^{12} - t^{13} + t^{13} + t^{14} + t^{16} + t^{17} + t^{18} - t^{26}$, which realises another family of K3 surfaces from~\cite{A}.
\end{eg}

These examples are not complete intersections in a weighted Grassmannian $(\tV\GIT_\chi\C^*) \cap H_1\cap\cdots\cap H_4$, for quasilinear hypersurfaces $H_i$, since there are no variables of weights one or two in $\chi$. To interpret these regular pullbacks as intersection, one can take a cone on the weighted Grassmannian, introducing additional variables of weights one and two, as in~\cite{CR,QS}. More general complete intersections inside weighted homogeneous spaces are also common. The way we define `format', taking hypersurface slices of one format descibes a new format, a tensor-like combination of the existing format and a complete intersection; see~\S\ref{sec!other}.

\begin{eg}\label{eg!zerowt}
There is no reason why format variables should be weighted positively. The role of the key variety is as a target for regular pullbacks, and these are defined on the affine cone, so there is no risk of taking $\Proj$ of a ring with non-positive weights.

For example, consider the same key variety $\CGr(2,5)\subset\C^{10}$ as above, but with key weights
\[
\chi=\begin{pmatrix}
0&1&1&1\\
&1&1&1\\
&&2&2\\
&&&2
\end{pmatrix}.
\]
A regular pullback to a nonsingular curve in $\PP^4$ defines a curve of genus five in its canonical embedding. If $\phi^*(x_1)=0$ then the curve is trigonal and lies on the scroll given by the minors of the upper $2\times 3$ block of the matrix. Deforming $\phi^*(x_1)=\lambda$ away from zero moves the regular pullback off the trigonal locus to give a non-special canonical curve, a $(2,2,2)$ complete intersection in $\PP^4$. This example can be extended to $\PP^5$, where the special pullback is the trigonal K3 surface extending this canonical curve.

In this format, the pullback by $\phi$ of the $5\times 5$ matrix is the matrix of first syzygies among the equations, so this matrix must not have non-zero constant entries, otherwise, as in the example, the free resolution is not minimal and we fall into a different format. Such entries only happen when the key weight is zero, and in that case we only remain in the format if the corresponding pullback is the zero polynomial, giving a special element of the family.

As another example, the weights
\[
\chi=\begin{pmatrix}
-1&1&1&1\\
&1&1&1\\
&&3&3\\
&&&3
\end{pmatrix}
\]
admit a regular pullback to a canonical surface in $\PP^5$, with $K^2=11$, $p_g=6$, where necessarily $\phi^*(x_1)=0$; as a sanity check, with these invariants Riemann--Roch gives
\[
P_X(t) = \frac{1 - 3t^2 + 2t^3 - 2t^4 + 3t^5 - t^7}{(1-t)^6}.
\]
For a general regular pullback this is just a degree $(3,4)$ complete intersection in $\PP^1\times\PP^2$ in the mild disguise of its Segre embedding, so is well-known, but there are other cases that cannot be expressed in such straightforward terms.

It is easy to see that one cannot allow two key weights $\le0$ that are pulled back to the zero polynomial. Below we note that even a single one cannot work for the kind of threefolds we seek. For example, attempting to make a quasismooth Calabi--Yau threefold with key weights
\[
\chi=\begin{pmatrix}
0&2&2&2\\
&2&2&2\\
&&4&4\\
&&&4
\end{pmatrix}
\]
and a regular pullback to $\PP(1,1,1,2,2,2,3)$, we find no problem when $\phi^*(x_1)\not=0$ except that $X$ is then a complete intersection rather than in this Grassmannian format, but when $\phi^*(x_1)=0$ the regular pullback is not quasismooth at the index three point.
\end{eg}

We seek threefolds, and in this format at least negative key weights do not arise:

\begin{prop}\label{prop!poswt}
Let $X$ be a variety in $\CGr(2,5)$ format with ambient weights $\chi$. If $X$ is of dimension $\ge 3$ and quasismooth then $\chi$ consists of strictly positive integers.
\end{prop}

\begin{proof}
%
%
%
If not, then without loss of generality $\phi^*(x_1)=0$ and any point of $X$ in the locus
\[
\left( \phi^*(x_8)=\phi^*(x_9)=\phi^*(x_{10})=0 \right) \subset X
\]
is a non-quasismooth point of embedding dimension at least four. This locus is necessarily non-empty if $\dim X\ge 3$.
\end{proof}

Note that the same conclusion holds when $X$ is a canonical threefold with arbitrary terminal singularities, since such terminal singularities have embedding dimension one by Reid~\cite{reidminimal} and Mori's~\cite{mori} classification.

\subsection{The Hilbert series of a canonical threefold}
Let $P=\frac{1}r(r-1,a,r-a)$ be a terminal quotient singularity with $r>1$ and $1\leq a<r$ coprime integers. (The first weight is $r-1$ since we consider varieties polarised by their canonical class.) Following~\cite{BRZ}, we define
\[
A=\frac{1-t^r}{1-t}=1+t+t^2+\cdots+t^{r-1}
\quad\text{ and }\quad
B=\prod_{b\in L}\frac{1-t^b}{1-t},
\]
and let $C=C(t)$ be the Gorenstein symmetric polynomial with integral coefficients such that $BC\equiv 1\ \modb{A}$ whose exponents lie in the integer range $\left\{\lfloor c/2\rfloor+1,\ldots,\lfloor c/2\rfloor+r-1\right\}$ (we abbreviate this to `$C$ is supported on $\left[ \alpha, \beta \right]$' for appropriate integers $\alpha, \beta$). In our case $X$ is a threefold with terminal singularities polarised by $K_X$, hence $c=5$.

\begin{thm}[\protect{\cite[Theorem~1.3]{BRZ}}]\label{thm!icecream}
Let $X$ be a canonical threefold with singularity basket $\cB$. For a terminal quotient singularity $Q=\frac{1}r(r-1,a,r-a)$, define
\[
\Porb(Q) = \frac{B(t)}{(1-t)^3(1-t^r)}
\]
where $B=B(t)$ is a polynomial supported on $\left[3, r + 1 \right]$ which satisfies
\[
B\times\prod_{b \in [r-1, a, r-a]}\frac{1-t^{b}}{1-t} \equiv 1 \mod \frac{1-t^r}{1-t}.
\]
Then the Hilbert series of $X$ polarised by $K_X$ is
\[
P_X = \Pini + \sum_{Q\in\cB} \Porb(Q),\quad\text{ where }\Pini = \frac{1+at+bt^2+bt^3+at^4+t^5}{(1-t)^4}
\]
for integers $a:=P_1-4$ and $b:=P_2 - 4P_1 + 6$.
\end{thm}

The relationship between $a$, $b$ and plurigenera $P_1$, $P_2$ is determined by the expansion
\[
P = 1 + P_1t + P_2t^2 + \cdots = 1 + (a+4)t + (b+4a+10)t^2 + \cdots,
\]
since each series $\Porb(t)$ has no quadratic terms or lower.

\begin{eg}\label{eg:1/2(1,1,1)}
Suppose that $p=\frac{1}{2}(1,1,1)$. We have $A=1+t$ and $B=1$, so the inverse of $B$ is $1$ modulo $A$. The numerator of $\Porb(p)$ is supported in the range $[3,3]$. Observe that $-t^3\equiv 1\ \modb{A}$, so
$$\Porb(p)=\frac{-t^3}{(1-t)^3(1-t^2)}.$$
Expanded formally as a power series, $\Porb(p)=-t^3-3t^4-7t^5-10t^6-\cdots$.
\end{eg}

\begin{eg}\label{eg:1/8(3,5,7)}
Suppose now that $p=\frac{1}{8}(3,5,7)$. Observing that
\begin{align*}
B&=(1+t+\cdots+t^6)(1+t+t^2)(1+t+t^2+t^3+t^4)\\
&\equiv -t^7(-t^3-t^4-t^5-t^6-t^7)(1+t+t^2+t^3+t^4)\\
&\equiv t^2(1+t+t^2+t^3+t^4)^2,
\end{align*}
where the equivalence is taken modulo $A=1+t+\cdots+t^7$, it is clear that
\begin{align*}
t^3(1+t^5+t^{10})(t^5+t^{10}+t^{15})B&\equiv t^5(1+t+\cdots+t^{14})(t^5+t^6+\cdots+t^{19})\\
&\equiv t^5\cdot t^{15}\cdot t^5\cdot t^{15}\\
&\equiv 1.
\end{align*}
So we have an inverse for $B$. To shift this inverse into the desired range of exponents (and hence find $C$) we use the fact that $t^8\equiv 1\ \modb{A}$:
\begin{align*}
t^3(1+t^5+t^2)(t^5+t^2+t^7)&\equiv t^3(t^5 + t^2 + t^7 + t^2 + t^7+t^4+t^7+t^4+t)\\
&\equiv t^3(-3-2t-t^2-3t^3-t^4-2t^5-3t^6).
\end{align*}
Thus
$$\Porb(p)=\frac{-3t^3-2t^4-t^5-3t^6-t^7-2t^8-3t^9}{(1-t)^3(1-t^8)}.$$
Until the final step all the polynomials appearing had non-negative coefficients. Since the last subtraction was required only to eliminate the out-of-range $t^7$ monomial, and since this monomial had the largest coefficient, we see that every coefficient of the numerator of $\Porb(p)$ is strictly negative. This is the case in general for canonically polarised terminal quotient singularities.
\end{eg}

\begin{thm}\label{thm!C_coeffs}
Let $X$ be a canonically-polarised threefold, and $p\in X$ be a terminal quotient singularity $\frac{1}{r}(-1,a,-a)$ for coprime integers $r>1$ and $1\leq a<r$. Define $m\in\Z$ by the conditions $0<m\leq r/2$ and $am\equiv -1\ \modb{r}$. Then
$$C(t)=c_3t^3+\cdots+c_{r+1}t^{r+1},$$
where
$$c_{i+3}=\begin{cases}
i_a-m&\text{if }0<i_a\leq m, \\
m-i_a&\text{if }m<i_a\leq 2m-1, \\
-m&\text{otherwise}.
\end{cases}$$
Here $0<i_a\leq r$ satisfies $i_a\equiv -im\ \modb{r}$. More concisely,
$$c_{i+3}=-\min\bigl\{m, \abs{m-i_a}\bigr\}.$$
\end{thm}

Notice that it might be necessary to switch the roles of $a$ and $-a$ in order for such an $m$ to exist---this is implicit in the statement of the theorem. For example, when considering Example~\ref{eg:1/8(3,5,7)} we are forced to take $a=5$.

Theorem~\ref{thm!C_coeffs} computes $\Porb$ for singularities of the form $Q=\frac{1}r(-1,a,-a)$. Multiplying by the natural denominator, the leading terms are
\[
(1-t)^3(1-t^r)\Porb(Q) = -mt^3 - \min\{m,r-2m\}t^4 -\cdots,
\]
where $m=-1/a\ \modb{r}$, as in the theorem.

\begin{cor}\label{cor:bound_on_num_sing}
Let $\Porb(p)=a_0+a_1t+a_2t^2+\cdots\in\Z[[t]]$ for some terminal quotient singularity $p\in X$. Then $a_0=a_1=a_2=0$ and $a_i<0$ for all $i\geq 3$. In particular there exists a bound on the number of singularities of $X$ in terms of $p_g$ and $P_2$.
\end{cor}

\begin{proof}[Proof of Theorem~\ref{thm!C_coeffs}]
With notation as above, observe that:
\begin{align*}
B&=(1+t+\cdots+t^{r-2})(1+\cdots+t^{a-1})(1+\cdots+t^{r-a-1})\\
&\equiv t^{r-1}(1+\cdots+t^{a-1})(t^{r-a}+\cdots+t^{r-1})\ \modb{A}\\
&=t^{2r-a-1}(1+t+\cdots+t^{a-1})^2.
\end{align*}
With $m$ as defined in the theorem,
$$t(1+t^a+t^{2a}+\cdots+t^{(m-1)a})(1+t+t^2+\cdots+t^{a-1})=t + t^2 + \cdots + t^{ma},$$
which is congruent to $-1$ modulo $A$. Hence:
\begin{align*}
C&\equiv t^{a+1}\cdot t^2(1+t^a+\cdots+t^{(m-1)a})^2\ \modb{A}\\
&=t^3(1+t^a+t^{2a}+\cdots+t^{(m-1)a})(t^a+t^{2a}+\cdots+t^{ma}).
\end{align*}
We consider the product of factors:
$$C_1=(1+t^a+t^{2a}+\cdots+t^{(m-1)a})(t^a+t^{2a}+\cdots+t^{ma}).$$

Recall that the numerator $C$ of $\Porb(p)$ is supported in $[3,r+1]$; we compute this by finding the integral polynomial equivalent to $C_1$ modulo $A$ supported in $[0,r-2]$.

The terms of $C_1$ arise as a product $t^{ja}$ with $0\leq j\leq m-1$ from the first factor and $t^{ka}$ with $1\leq k\leq m$ from the second. Hence the coefficient of $t^{ia}$ in the resulting expansion is given by:
$$\begin{cases}
i,&\text{if }0<i\leq m;\\
2m-i,&\text{if }m<i\leq 2m-1.
\end{cases}$$
Since $a$ is coprime to $r$, the resulting monomials are equivalent modulo $1-t^r$ (and hence also modulo $A$) to distinct powers of $t$ in the range $t,\dots,t^{r-1}$ (recall that by definition $2m-1\leq r-1$). We obtain the equivalent polynomial:
$$C_1\equiv c'_1t+\cdots+c'_{r-1}t^{r-1}\ \modb{A},$$
where:
$$c'_i=\begin{cases}
i_a,&\text{if }0<i_a\leq m;\\
2m-i_a,&\text{if }m<i_a\leq 2m-1;\\
0,&\text{otherwise.}
\end{cases}$$
Subtracting $mA$ from this (to shift the degree down by one) gives the desired result.
\end{proof}

\section{Enumeration of Hilbert series and varieties}
We aim to construct $d$-dimensional varieties $X\subset \PP(W)$, for weights $W$, in a given format and with canonical class $\omega_X = \cO_X(k)$ for given $k$. Moreover we insist that the singularities appearing on $X$ are those of some chosen family. This could be a meaningful complete family---terminal threefold singularities, say---or an arbitrary collection amenable to computation---isolated four-fold terminal quotient singularities, for example. In either case, we have to be able to compute their $\Porb$.

\subsection{The general process to find orbifolds}
Fix a key variety $\tV\subset\C^n$ of codimension~$c$, and fix integers $d, k\in\Z$ with $d\ge 2$ and a class of singularities $Q$ for which $\Porb(Q)$ is computable. We aim to construct $d$-dimensional varieties $X$ in weighted projective space that have $K_X=\cO_X(k)$, singularities in the chosen class, and key variety~$\tV$. This pseudo-algorithm is similar in spirit to that of~\cite{CR} and~\cite{QS}, but differs in that here we determine the target Hilbert series first and then try to match a basket, rather than choosing a basket and computing the Hilbert series.

\begin{enumerate}[leftmargin=25pt]
\item\label{alg!grading}
Choose a grading $\chi$ on $\tV$. This determines a format $F = (\tV,\chi,\F)$.

\item\label{alg!weights}
List all possible ambient weights $W$ for which there is a map $\phi\colon\C^{d+c+1}\rightarrow \C^{n}$ that is equivariant of degree zero with respect to the diagonal $\C^*$ action with weights $W$ in the domain and $\chi$ in the codomain; that is, $\phi$ is defined by a vector of $n$ polynomials homogeneous with respect to $W$ of weights exactly $\chi$ (and not a multiple of $\chi$).

\item\label{alg!hilbert}
Setting $\widetilde{X} = \phi^{-1}(\tV)$, write out the Hilbert series $P_X(t)$ of $X=\widetilde{X}\GIT_{\! W}\,\C^*\subset\PP(W)$, and determine the initial term $\Pini(t)$.

\item\label{alg!basket}
Set $R(t) = P_X(t) - \Pini(t)$. Compute all ways of realising $R(t) = \sum_{Q\in\cB} \Porb(Q)$ for finite sets $\cB$ of singularities of the chosen family. If there are no solutions, then a variety cannnot be realised admitting only the given class of singularities.

\item\label{alg!oracle}
Accept or reject candidate Hilbert series according to whether or not there exists an orbifold in the given format that realises it.
\end{enumerate}
Apart from the final step~\eqref{alg!oracle}, this process can be automated on any computer algebra system---it uses only standard tools such as rational functions and power series. Steps~\eqref{alg!grading} and~\eqref{alg!oracle} rely on knowledge of the chosen format. The other steps are essentially independent of the format, and we discuss these first.

\subsubsection{Step~\eqref{alg!weights}\emph{:} Enumerating the ambient weights}
The maximum key weight $\chi_{\mathrm{max}}$ is part of the format. For orbifolds (or canonical threefolds with terminal singularities) no variable can be omitted from the equations, so the largest degree occurring in any ambient weight sequence $W$ cannot exceed $\chi_{\mathrm{max}}$. Together with the condition that $\sum_{a\in W} a = k - k_{\tV}$, this implies that there are only finitely many weight sequences $W$, and they can easily be computed with standard techniques. (One can immediately reject sequences that will lead to non-well-formed varieties, for example when $W$ has a nontrivial common divisor.)

\subsubsection{Step~\eqref{alg!hilbert}\emph{:} Recovering the Hilbert series $P_X$ and $\Pini$}
For each choice of $\chi$ and of $W$, we suppose that suitable regular pullback $\phi$ exists, and write $P_X(t)$ using the formula of Proposition~\ref{prop!PX}. As power series expansions, the $\Porb$ summands have terms that start in degree $\lfloor d+k+1 \rfloor + 1$, so that $\Pini$ agrees with $P_X$ in all degrees up to its centre of Gorenstein symmetry. So to compute the numerator of $\Pini$ we need only determine whether any equations have low degrees and compensate appropriately in the corresponding coefficients of $P_X$. For canonical threefolds, the coefficients of $t$ and $t^2$ are enough.

\subsubsection{Step~\eqref{alg!basket}\emph{:} Polytopes and knapsack kernels}\label{sec!knapsack}
Next we match the possible $\Porb$ contributions arising from the candidate singularities $\sigma_1,\ldots,\sigma_m$ to the Hilbert series, and so build the possible baskets. This is a ``knapsack''-style search: summing non-negative multiples of a known collection of vectors to obtain a given solution. The first few terms of each possible $\Porb$ contribution, together with the target sequence $P_X-\Pini$, are used to construct a polyhedron in the positive orthant whose integer points $(a_1,\ldots,a_m)\in\Z_{\geq 0}^m$ give solutions to $\sum a_i\Porb(\sigma_i)=P_X-\Pini$. It is an important point that the resulting polyhedron may be infinite: it decomposes into a sum of a compact polytope $Q$ and a (possibly empty) tail cone $C$. The points in $Q$ correspond to the possible baskets for $X$, whilst the Hilbert basis of $C$ describes the possible ``kernels''; that is, collections of singularities whose net $\Porb$ contribution is zero, so can be added to any basket.

\subsubsection{Remarks}
The process described above does not even in principle give rigorous classification results---the key varieties we use have infinitely many diagonal $\C^*$ actions. It is worth being clear about where the process is finite and determined, where it is infinite but under control, and where it contains essentially infinite searches.

\begin{enumerate}[leftmargin=25pt]
\item
The ambient weights $W$ are solutions to a ``knapsack''-type problem---find a fixed number of strictly positive integers with a given sum. Such problems of course have a finite solution, with well-documented algorithms, if one wants to implement them.

Our approach has a striking virtue: it is easier to solve for ambient weights $W$ if one imposes additional conditions on the weights than if one does not. For example, to find cases of canonical threefolds with empty bi-canonical linear system we can solve for $W$ among integers $\ge 3$. Such conditions dramatically simplify the problem; compare~\S\ref{sec!more}.


\item
As explained in~\S\ref{sec!knapsack}, the list of possible baskets that solve the purely numerical problem of completing $\Pini$ to the Hilbert series $P_X$ can be infinite. But even then, it is represented by the points of a finitely-determined polyhedron, and these points can be enumerated in a systematic order, from `small' baskets to `large' baskets. Any given candidate variety has known ambient weights and equation degrees, and so only finitely many of these baskets could possibly occur.

The kind of elementary calculation one faces is this: if the ambient stratum that has an index three stabiliser is $\Gamma=\PP(3,6)$, and if one of the equations has degree twelve, then, unless the format forces this equation to vanish along $\Gamma$, there cannot be more than two orbifold points of index three, since this equation restricted to $\Gamma$ is quadratic.

\item
Although many geometrically important searches will have a finite solution (compare~\cite[Theorem~4.1]{JK} for quasismooth hypersurfaces), the search routine outlined above does not have a stopping condition and we cannot know if or when all solutions have been found.
This is in the same spirit as Iano--Fletcher's original enumeration for Fano threefolds in codimension two (retrospectively complete by~\cite{CCC}), but differs from Reid's computation of the $95$ Fano hypersurfaces and Johnson--Koll\'ar's calculation of Fano complete intersections.
For many of our searches, we simply continue searching until no new results appear; see the columns $k_{\mathrm{last}}$ and $k_{\mathrm{max}}$ of Table~\ref{tab!summary}.

\item
The process as stated works in any generality for any key variety. We describe the $\Gr(2,5)$ format in detail in~\S\ref{sec!Gr_2_5}, and sketch some other formats in~\S\ref{sec!other}.

\item
We have not used the condition that $\phi$ exists except to bound the weights appearing in $W$, nor have we enforced the condition that $\phi^{-1}(\tV)$ is Cohen--Macaulay. Both of these are postponed to the final step.
\end{enumerate}

\subsection{Canonical threefolds in $\Gr(2,5)$ format}\label{sec!Gr_2_5}
We make formats with the codimension three key variety $\tV=\CGr(2,5)$ of Example~\ref{sec!pfaff} and its usual Pfaffian free resolution.

\subsubsection{Steps~\eqref{alg!grading}--\eqref{alg!basket}}
Iterating over the possible gradings $\chi$ is one pass through an infinite loop. By~\cite{CR}, $\chi$ is determined by a vector $(w_1,\dots,w_5)$ with either all $w_i\in\Z$ or all $w_i\in \frac{1}2+\Z$: for Pl\"ucker coordinates $x_{ij}$ with $1\le i<j\le5$, set $\deg x_{ij} = w_i+w_j$, and then $\chi = (\chi_{ij})$. To enumerate all possible $w$, we may assume $w_1\le\cdots\le w_5$. By Proposition~\ref{prop!poswt}, when $d\ge3$ all key variables have positive degrees, so $w_1+w_2>0$, and in particular $w_2>0$. The adjunction number of the key variety is $k_{\tV} = 2\sum w_i$. A naive search routine now computes all $w$ satisfying these conditions for a given $k_{\tV}$ (which is finite), and the full search is carried out in increasing adjunction number $k_{\tV}=1,2,\dots$; this is the only point where the search is not finite.

The weights of the five equations, $d_j=(\sum w_i)-w_{6-j}$, are determined by the format and satisfy $d_1\le\cdots\le d_5$. For Step~\eqref{alg!weights} we choose weights $a_0\le \cdots \le a_6$ of a potential ambient space $\PP(a_0,a_1,\dots,a_6)$. To find canonical varieties we choose $\sum a_i = k - 1$.

If $X\subset\PP(a_0,a_1,\dots,a_6)$ is a variety in this format, then its Hilbert series is $P_X(t) = \Pnum/\Pi$ where $\Pi := \prod (1 - t^{a_i})$ and
\[
\Pnum := 1 - t^{d_1} - \cdots - t^{d_5} + t^{k-d_5} + \cdots + t^{k-d_1} - t^k
\]
with $k=2\sum w_i$.

It is easy to see that for canonical threefolds there will be no equations of degree two, and so the first two coefficients of the power series expansion $P_X = 1 + P_1t + P_2t^2 + \cdots$ are $P_1 = c_1$ and $P_2 = c_2 + \frac{1}{2}c_1(c_1+1)$, where $c_s$ is the number of $a_i$ equal to $s$.

\subsubsection{Step~\eqref{alg!oracle}\emph{:} Complete intersections in cones}
In practice it is often convenient to treat candidate varieties as complete intersections inside projective cones, even though the regular pullbacks we use can be more general. If possible we apply Bertini's theorem. However, when there are many different weights bigger than one, the base loci appearing in successive ample systems tend to be large.

\begin{eg}
\emph{Number~4 in Table~\ref{tab!codim3}: $X\subset\PP(1^5,2^2)$.}
Let $V_1\subset\PP(1^5,2^{10})$ be the projective cone over $\tV$ with vertex $\PP^4$, which is also the locus of non-quasismooth points. Then $X\subset V_1$ is the complete intersection of eight quadrics. The system of quadrics has empty base locus, and between them they miss the vertex, so $X$ is quasismooth by Bertini's theorem.
\end{eg}

Numbers~$1$ and~$2$ in Table~\ref{tab!codim3} work in the same way: the complete intersection in the end has empty base locus because there are no coprime weights to be eliminated.

\begin{eg}
\emph{Number~6 in Table~\ref{tab!codim3}: $X\subset\PP(1^4,2^2,3)$.}
Let $V_1\subset\PP(1^4,2^3,3^4,4)$ be the projective cone over $\tV$ with vertex $\PP^1$. Consider $V_2\subset V_1$, a general complete intersection of three cubics. Between them, these cubics miss $V_1\cap\PP(3^4)$, since that is codimension one in $\PP(3^4)$, and they miss the vertex too. But each cubic does have base locus $V_1\cap\PP(2^3,4)$, which is codimension one in $\PP(2^3,4)$, and is in fact a surface together with residual point. So at this stage we know that $V_1\subset\PP(1^4,2^3,3,4)$ is quasismooth away from that locus. (Eliminating the variables does not cause confusion, since the locus of concern is exactly where they all vanish, and so it doesn't move away from $\PP(2^3,4)$ when we eliminate---that is obvious in this case, since that is the only stratum with any index two stabiliser, but we need to know this in other situations later too.)

Now let $V_3\subset V_2$ be the locus of a general quartic. The linear system of quartics has base locus $V_2\cap\PP(3^4)$, but that is empty. So $V_3\subset\PP(1^4,2^3,3)$ is quasismooth away from a curve $\Gamma\subset\PP(2^3)$. Finally $X\subset V_3$ is the locus of a general quadric. The system of quadrics has empty base locus on $V_3$, so the only question remains about the point(s) where the quadric vanishes on~$\Gamma$. But it is easy to write equations for a specific $X$ that meets $\PP(2^3)$ in a single point that is manifestly quasismooth, and so the general $X$ is quasismooth as claimed.
\end{eg}

Numbers~3, 5 and~7--11 in Table~\ref{tab!codim3} work in the same way: each new hypersurface cuts the existing base locus down, but there is new base locus to consider too.

\begin{eg}
\emph{Number~12 in Table~\ref{tab!codim3}: $X\subset\PP(1^2,2^2,3^2,4)$.}
Let $V_1\subset\PP(1^2,2^2,3^3,4^4,5)$ be the projective cone over $\tV$ with vertex $\PP^1$. The final variety $X$ will simply be a $3,4,4,4,5$ complete intersection in $V_1$, but Bertini's theorem is not so easy to apply since most low-degree linear systems have rather large base locus. Nevertheless, with care it can still be made to work.

First consider $V_2\subset V_1$, a general complete intersection of three quartics. Between them, these quartics miss $V_1\cap\PP(4^4)$, since that is codimension one there, and they miss the vertex too. But each quartic does have base locus $V_1\cap\PP(3^3,5)$, which is a copy of $\PP(3^2,5)$ and a residual index three point. (So far similar to the previous example.)

Now let $V_3\subset V_2$ be the locus of a general quintic. It meets the previous base locus in $V_2\cap\PP(3^3)$---a line and a disjoint point---and it also has base locus of its own, namely
\[
\left( V_2\cap\PP(2^2,4) \right) \cup \left( V_2\cap\PP(3^3,4) \right).
\]
We leave the first of these for now, but note that the second is a collection of finitely many points, none of which are at the index four point. At this stage we have $V_3\subset\PP(1^2,2^2,3^3,4)$, with the three groups of loci of concern.

Finally $X\subset V_3$ is the locus of a general cubic. It misses all isolated base points, other than those lying in $\PP(2^2,4)$, and cuts the index three line in a single point; calculation on an example shows this point to be $\frac{1}3(1,2,2)$ in general.

It remains to consider the locus $V_3\cap\PP(2^2,4)$, since this is in the base locus of the linear system of cubics. Calculation on an example shows that this is finitely many $\frac{1}2(1,1,1)$ points, and a standard weighted Hilbert--Burch calculation confirms that there are four such points (necessarily, from the original orbifold Riemann--Roch calculation, if you prefer).
\end{eg}

One could continue, but the calculations become rather fiddly, with many distinct base loci to keep track of. We settle, at this stage, for computing random examples defined over the rational numbers and using computer algebra to check that their Jacobian ideals define the empty set. For example, number~18 in Table~\ref{tab!codim3}, $X\subset\PP(3,4^2,5^2,6,7)$, can be realised by the Pfaffians of the skew $5\times 5$ matrix
\[
\begin{pmatrix}
y&t&v&w\\
&v&w&xt+y^2+z^2\\
&&xu+yz&x^3\\
&&&t^2+u^2
\end{pmatrix}.
\]

\subsubsection{Plurigenus invariants}
We recall the plurigenus formula:

\begin{thm}[\protect{\cite[Theorem~5.5]{C3f},~\cite[Theorem~2.5(4)]{F87}}]
\label{thm!pluri}
Let $X$ be a canonical threefold with singularity basket $\cB$ and $\chi=\chi(\cO_X)$. Then
\[
h^0(X,mK_X) = (1-2m)\chi + \frac{m(m-1)(2m-1)}{12}K^3 + \sum_{p\in\cB} c_m(P)
\]
where, for $P=\frac{1}r(-1,a,-a)$ and $ab\equiv 1\ \modb{r}$, we have
\[
c_m(P) = \sum_{i=1}^{m-1} \frac{\overline{ib}(r-\overline{ib})}{2r}.
\]
\end{thm}

Iano--Fletcher~\cite{F87} gives four different expressions for the terms in the plurigenus formula. In fact, this formula holds exactly as stated for any projective threefold with canonical singularities. The plurigenus formula goes together with the Barlow--Kawamata formula~\cite{K86} for $K_X \cdot c_2(X)$:
\[
\pi^*K_X \cdot c_2(Y) = \sum_Q \frac{r^2-1}{r} - 24\chi(\cO_X),\quad\text{ for any resolution }\pi\colon Y\rightarrow X.
\]

\begin{cor}[Basic numerology]
Set $P_m = h^0(X,mK_X)$ for $m\in\Z$. It follows from Kawamata--Viehweg vanishing that
\[
P_m = \chi(X,mK_X),\quad\text{ for }m\ge 2,
\]
and from Theorem~\ref{thm!GW} that $h^1(X,K_X) = h^2(X,\cO_X) = 0$ and $h^2(X,K_X) = h^1(X,\cO_X) = 0$, so that
\[
P_1 = \chi(X,K_X) + 1,\quad\text{ or equivalently that }\chi(X,\cO_X) = 1 - P_1.
\]
\end{cor}

We use the plurigenus formula to calculate $K_X^3$ and $K_X\cdot c_2(X)$ in Tables~\ref{tab!codim3} and~\ref{tab!codim5}.

\section{Other formats and varieties}
\subsection{Other formats}\label{sec!other}
We can consider any affine Gorenstein variety that admits some $\C^*$ actions to be a Gorenstein format, following Reid~\cite[1.5]{fun}, so there are very many. We list a few, including those that appear in Table~\ref{tab!summary} below. The point $\tV = V(x_1=\cdots=x_n=0)\subset\C^n$ is a key variety, and regular pullbacks from formats based on this are complete intersections. Qureshi and Szendr{\H o}i~\cite{QS} use quasihomogeneous varieties for Lie groups as formats, extending those of Corti and Reid~\cite{CR}. Other formats that often arise in practice for varieties in codimension four are included in~\cite[\S9]{TJ}; the rolling factors format is described by Stevens~\cite{stevens01}, and is used by Bauer, Catanese and Pignatelli~\cite{BCP} to construct surfaces of general type.

We can take products of formats to make new ones. Given two formats
\[
\tV = V(f_1,\dots,f_s) \subset \C^n
\]
with key weights $\chi=(\chi_1,\dots,\chi_n)$ and Hilbert numerator $N(t)$, and
\[
\widetilde{U} = V(g_1,\dots,g_r) \subset \C^m
\]
with key weights $\psi=(\psi_1,\dots,\psi_m)$ and Hilbert numerator $M(t)$, we can make a format
\[
\widetilde{W} = V(f_1,\dots,f_s,g_1,\dots,g_r) \subset \C^{n+m}
\]
with key weights $(\chi_1,\dots,\chi_n,\psi_1,\dots\psi_m)$ and Hilbert numerator $N(t)\times M(t)$. (We omit the free resolution information here, since we do not need it for the calculations in Table~\ref{tab!summary}.)

For example, the product of $\Gr(2,5)$ and a codimension one complete intersection describes (non-quasilinear) hypersurfaces inside weighted Grassmannian pullbacks, which have six equations and ten first syzygies; in Table~\ref{tab!summary} we denote this format by $\Gr(2,5)\cap H$. Non-special canonical curves of genus six are in this format.

\subsubsection{Orthogonal Grassmannian in codimension five}
We recall the weighted orthogonal Grassmannians of~\cite{CR}, and we list canonical threefolds in this format in Table~\ref{tab!codim5}.

Let $w=(w_1,\dots,w_5)$ as above ($w_i$ all congruent modulo~$\Z$ and have denominator one or two) and positive $u\in\Z$. These parameters will determine certain weights. There are sixteen indeterminates: $x$, $x_1,\dots,x_5$, and $x_{ij}$ for $1\le i<j\le 5$. The ten equations are
\[
x x_i = \Pf_i(M)\qquad\text{ and }\qquad M (x_1,\dots,x_5)^t= (0,\dots,0)^t,
\]
where $M$ is the antisymmetric $5\times 5$ matrix with upper-triangular entries $x_{ij}$, and the signed maximal Pfaffians $\Pf_1(M),\dots,\Pf_5(M)$ of $M$ are
\[
\Pf_i(M) = (-1)^i ( x_{jk}x_{lm} - x_{jl}x_{km} + x_{jm}x_{kl}),
\]
where $\{i,j,k,l,m\} = \{1,\dots,5\}$ and $j<k<l<m$.

These equations are homogeneous with respect to the weights
\[
\wt x = u,\qquad\wt x_i = u+|w|-w_i,\qquad\wt x_{ij}=w_i+w_j+u,
\]
so the ten equations respectively have weights
\[
2u+|w|-w_i\quad\text{ and }\quad 2u+|w|+w_i,\quad\text{ for }i=1,\dots,5.
\]
We may assume that $u=\wt x$ is smallest weight in the format and that $w$ is ordered; these are normalising conditions to prevent duplication of the same format (up to automorphism) for different choices of $u$ and $w$. We enforce that $w_i+w_j> 0$ for all $i$, $j$; in particular, only $w_1$ may be negative.

The ten equations define $\tV=C\!\OGr(5,10)$, the affine cone over the orthogonal Grassmannian; the weights determine a $\C^*$ action on $\tV$. We do not need to know more of the free resolution of the coordinate ring---in the given order, the Jacobian matrix is the matrix of first syzygies---except to note the canonical degree $k$ which is
\[
k_{\tV} = 4|w|+8u.
\]
The first example in Table~\ref{tab!codim5} appears as~\cite[Example~5.1]{CR}. Arguing with Bertini's theorem shows that the first five entries of the table really do exist as claimed. The argument becomes more involved, and we have not verified the remaining cases---although they do intersect the orbifold loci correctly---so they should be treated only as plausible candidates.

\subsection{Other classes of variety}\label{sec!more}
The method we describe can construct examples of other classes of varieties. We describe a few other classes briefly to give an idea both of the flexibility of our approach and its limitations.  Our results are summarised in Table~\ref{tab!summary}; the results themselves are available online via the Graded Ring Database~\cite{grdb}.

\begin{table}[h]
\[
\begin{array}{ccccccccc}
\toprule
\dim & k & \text{codim} & \text{Format} & \text{Reference} & k_{\mathrm{last}} & k_{\mathrm{max}} & \text{\#raw} & \text{\#results} \\
\cmidrule(lr){1-2}\cmidrule(lr){3-5}\cmidrule(lr){6-8}\cmidrule(lr){9-9}
\evnrow 3 & -1 & 1 & \text{c.i.} & \text{\cite{Fletcher}} & 66  & 90 & 95  & 95 \\
\evnrow  &  & 2 & \text{c.i.} & \text{\cite{Fletcher}} & 54 & 124 & 85 & 85 \\
\evnrow && 3 & \text{c.i.} & \text{classical} & 6 & 77 & 1 & 1 \\
\evnrow && 3 & \Gr(2,5) & \text{\cite{A}} & 45 & 70 & 69 & 69 \\
\evnrow && 4 & \Gr(2,5)\cap H & \text{classical} & 7 & 45 & 1 & 1 \\
\evnrow && 5 & \OGr(5,10) & \text{classical} &4&73 & 1 & 1 \\
\oddrow 3 & 0 & 1 & \text{c.i.} & \text{\cite{KS}} &  &  &  & 317 \\
\oddrow  &  & 2 & \text{c.i.} & & 120 & 121 & 419 & 401 \\
\oddrow && 3 & \text{c.i.} & & 74 & 77 & 25 & 22 \\
\oddrow && 3 & \Gr(2,5) & & 71 & 71 & 226 & 187 \\
\oddrow && 4 & \text{c.i.} &\text{classical} & 8 & 32 & 1 & 1 \\
\oddrow && 4 & \Gr(2,5)\cap H & & 39 & 46 & 123 & 14 \\
\oddrow && 5 & \OGr(5,10) &  &44&46 & 23 & 23 \\
\evnrow 3 & 1 & 1 & \text{c.i.} & \text{\cite{Fletcher}} & 46 & 85 & 23 & 23 \\
\evnrow  &  & 2 & \text{c.i.} & \text{\cite{Fletcher}} & 40 & 130 & 66 & 59 \\
\evnrow && 3 & \text{c.i.} & \text{\cite{Fletcher}} & 46 & 80 & 38 & 37 \\
\evnrow && 3 & \Gr(2,5) & & 35 & 71 & 18 & 18 \\
\evnrow && 4 & \text{c.i.} &\text{classical} & 9 & 34 & 1 & 1 \\
\evnrow && 4 & \Gr(2,5)\cap H & & 41& 46 & 84 & 57 \\
\evnrow && 5 & \text{c.i.} &\text{classical} & 10 & 30 & 1 & 1 \\
\evnrow && 5 & \OGr(5,10) &  &32&74 & 21 & 21 \\
\bottomrule
\end{array}
\]
\caption{The number of cases of Fano, Calabi--Yau, and canonical orbifolds in various formats. All were computed allowing isolated canonical quotient singularities. The column $k_{\mathrm{last}}$ gives the largest adjunction number for which a result was found; $k_{\mathrm{max}}$ gives the largest degree searched;  \#raw gives the number of candidates found by the computer;  \#results gives the number of candidates after removing obvious failures.
\label{tab!summary}}
\end{table}

\subsubsection{Understanding Table~\ref{tab!summary}}
The table is generated by a systematic computer search in order of increasing adjunction number $k$.
The search continues until the calculations become unwieldy.
The table indicates this point: $k_{\mathrm{max}}$ is the largest adjunction number up to which the search is complete.
It also records largest adjunction number, denoted~$k_{\mathrm{last}}$, for which a candidate was found.
Table~\ref{tab!summary} records the number of cases found, denoted \#raw.
In a few case, it is easy to see that there cannot be a quasismooth realisation of a candidate.
For example, any 3-fold
\[
X_{6,30} \subset \PP(1,2,3,4,10,15),
\]
must have a singularity at its intersection with the line $\PP(10,15)$:  the equation of degree~6 cannot
provide a tangent monomial there.
The final column \#results records the number of results after removing such cases that obviously fail.

When $k_{\mathrm{max}}$ is much larger than $k_{\mathrm{last}}$---for example, canonical 3-folds in codimensions~3 and~5---it is conceivable that we have found all the results. When the two numbers are close, almost certainly we are only part of the way through the complete list.
For example, a general codimension~4 variety $X\subset \PP(4,5,6,6,7,7,8,9)$ defined by an equation of degree~18 and the maximal Pfaffians of a $5\times 5$ antisymmetric matrix with degrees
\[
\begin{pmatrix}
4 & 5 & 6 & 7\\ &6 & 7 & 8 \\ &&8 & 9\\ &&&10
\end{pmatrix}
\]
is a quasismooth canonical 3-fold with adjunction number $k=53$, which exceeds $k_{\mathrm{max}}$ in this case, and so does not appear in the results on \cite{grdb}.

\subsubsection{Calabi--Yau threefolds}\label{s!CY3}
This is the case $K_X=0$ and $h^1(X,\cO_X)=0$ (the latter being automatic for Gorenstein formats by Theorem~\ref{thm!GW}). One could insist that $X$ be a manifold, but we search more widely amongst orbifolds $X$ with canonical singularities. Since we construct Gorenstein rings, and work on 3-fold orbifolds, the canonical singularities that arise all admit a crepant resolution, so that resulting orbifolds $X$ have a resolution of singularities that is a Calabi--Yau manifold.

There are lots of Calabi--Yau threefolds, and in each format we are still discovering examples at the point that the computer search becomes unreasonably slow. In $\Gr(2,5)$ format, Table~\ref{tab!summary} shows that the search was completed in full up to the value $k_{\mathrm{max}}=71$ of this parameter, but it also shows that the last example was at the value $k_{\mathrm{last}}=71$ of the parameter. No doubt there will be more cases for higher values of the parameter.

Some candidates cannot be realised by an orbifold; these are removed from the raw lists by hand.
In most cases, their failure to be quasismooth occurs on the orbifold loci, so is easy to see.
However, there are a few that are fine at the orbifold locus but singular elsewhere.
For example, $X\subset\PP(1,1,2,5,8,13,19)$ defined with syzygy degrees
\[
\begin{pmatrix}
1 & 2 & 6 & 8\\ &6 & 12 & 14 \\ &&13 & 15\\ &&&19
\end{pmatrix}
\]
must contain the coordinate plane $D=\PP(5,13,19)$: the first two rows and columns
of this matrix necessarily lie in the ideal $I_D$, for reasons of degree.
Any general such threefold $X$ is still a Calabi--Yau 3-fold, but is not $\Q$-factorial,
and has single node lying on $D$.
In the terminology of \cite{TJ}, $D\subset X$ is in Jerry$_{12}$ format, and
following the methods there it can be unprojected to give a quasismooth Calabi--Yau
3-fold $Y\subset\PP(1,1,2,5,8,13,19,37)$, embedded in codimension~4,
with a single $\frac{1}{37}(5,13,19)$ orbifold point: the birational map $X\dashrightarrow Y$
is the small $D$-ample resolution of the node followed by the contraction of $D$
to the orbifold point.
Unlike the results of \cite{TJ} and \cite{Geo}, $X$ cannot be deformed to quasismooth
in its Pfaffian format: $D\subset X$ always appears as Jerry$_{12}$, and $Y$ is
only realised as one deformation family. (As mentioned in \cite{TJ}, Jerry tends
to have higher degree than Tom, so having Jerry with just one node makes it hard for
Tom to exist.)

\subsubsection{Comparison with known lists: the famous 95 and all that}
We recalculated the known classifications of Fano threefolds that arise in the formats we compute. The classical Fano threefolds of Table~\ref{tab!summary} can be found in~\cite{IP}. The famous $95$ hypersurfaces of~\cite{C3f}, the $85$ codimension two complete intersections of Iano--Fletcher~\cite{Fletcher}, and Alt{\i}nok's $69$ codimension three $\Gr(2,5)$ cases all appeared early in their respective searches. (If run for K3 surfaces, the trigonal K3 surface of Example~\ref{eg!zerowt} also appears.)
We find the classical $X_{2,2,2}\subset\PP^6$ in codimension~3, and \cite{CCC} prove that there
are no more Fano complete intersections.
Although we do not list them in the table, we also checked Suzuki's index two Fano threefolds: $26$ in codimension two and two in codimension three in \cite{BS07} Tables~2 and 3.

In higher codimensions, there will be many different formats, and any single format is likely to
realise only a few of the possible varieties.
In codimension~4, \cite{grdb} lists 145 Hilbert series of Fano 3-folds, whereas
the $6\times 10$ codimension~4 format of \S\ref{sec!other} realises only a single family.
The remaining 144 do exist, usually as two or more families: see \cite{Pap,TJ}.
In codimension~5, again the format we demonstrate realises a single family, while \cite{grdb}
lists 164 possible Hilbert series.

Canonical threefolds that arise as complete intersections appear in \cite{Fletcher}, and those lists are proved complete in \cite{CCC}; in particular, there are no examples in codimension~6 or higher. The codimension two and three complete intersections we find include some interesting near misses. Seven of the raw results are elliptic fibrations over rational surfaces, so not of general type, and we removed these by hand (see the columns \#raw and \#results in Table~\ref{tab!summary}). Each one has a hyperquotient singularity of type $\frac{1}{4}(1,1,2,3;2)$ that is not terminal---but it takes more than numerical data to see that.

\subsubsection{Hypersurfaces}
Complete intersections in codimension one illustrate the limitations of this approach. Although we find the famous $95$ easily, there are, also famously, $7555$ quasismooth Calabi--Yau hypersurfaces, of which $317$ have isolated quotient singularities. In theory the algorithm will eventually find all of these $317$ cases, but in practice our code finds only the first $194$ of them before becoming unreasonably slow;
we include this case in Table~\ref{tab!summary} for completeness, but did not calculate it using this method.

There are other specialised algorithms that handle hypersurfaces
more effectively.
To find all 7555 independently of~\cite{KS}, one can use
the well-known `quasismooth hypersurface' algorithm of~\cite{C3f,JK}
that we implement in~\cite{BK}. That algorithm does not require the
singularities to be isolated, but analyses all singular loci.

\subsubsection{Higher index threefolds of general type: the case $\chi = 1$}
The same methods apply to varieties polarised by a Weil divisor $A$
which satisifies $K_X = kA$ for some $k>1$.
Regular canonical threefolds with $\chi > 0$, or equivalently $h^0(X,K_X)=0$, are
fairly rare, but we
can search for them directly by using weights $W$ that do not
include $1$ (or $2,3,\dots$).

For example, setting $k=2$, so that $K_X=2A$, we find
\[
X_{18,35}\subset\PP(5,6,7,9,11,13)
\hbox{ with }
\begin{cases}
P_1=P_2=0, P_3=1\\
\cB = \left\{\frac{1}{3}( 1, 1, 2 ),\frac{1}{11}( 5, 6, 9 ),\frac{1}{13}( 6, 7, 11 )\right\}\\
K_X^3=8/429.
\end{cases}
\]
An example with $K_X=3A$ is given by
\[
X_{60}\subset\PP(4,5,7,11,30)
\hbox{ with }
\begin{cases}
P_1=P_2=0 \hbox{ and $S\in|3K_X|$ is not irreducible}\\
\cB =\left\{\frac{1}{2}( 1, 1, 1 ), 2 \times\frac{1}{5}( 1, 2, 4 ), \frac{1}{7}( 2, 4, 5 ),\frac{1}{11}( 4, 7, 8 )\right\}\\
K_X^3=27/770,
\end{cases}
\]
and similarly with $K_X=4A$ by $X_{42} \subset\PP(5,6,7,9,11)$,
which manages $P_2=0$ despite having three variables in degree $<8$.

\subsection*{Acknowledgements}
This work was supported in part by EPSRC grant EP/E000258/1. The second author is supported by EPSRC grant EP/I008128/1.

\begin{landscape}
\begin{longtable}{>{\hspace{0.5em}}llccccr<{\hspace{0.5em}}}
\caption{Codimension three.} \label{tab!codim3}\\
\toprule
\multicolumn{1}{c}{Variety}&\multicolumn{1}{c}{Basket $\mathcal{B}$}&\multicolumn{1}{c}{$K^3_X$}&\multicolumn{1}{c}{$\chi$}&\multicolumn{1}{c}{$K_Xc_2$}&\multicolumn{1}{c}{$w$}&\multicolumn{1}{c}{Syz weights}\\
\cmidrule(lr){1-1}\cmidrule(lr){2-2}\cmidrule(lr){3-5}\cmidrule(lr){6-6}\cmidrule(lr){7-7}
\endfirsthead
\multicolumn{7}{l}{\vspace{-0.25em}\scriptsize\emph{\tablename\ \thetable{} continued from previous page}}\\
\midrule
\endhead
\multicolumn{7}{r}{\scriptsize\emph{Continued on next page}}\\
\endfoot
\bottomrule
\endlastfoot
\evnrow $\begin{array}{@{}l@{}}\evnrow X_{3^4, 4}\\\quad\subset\PP(1^7)\end{array}$&&$20$&$-6$&$144$&$(0, 1, 1, 1, 1)$&$\footnotesize\begin{matrix} 1&1&1&1\\ &2&2&2\\ &&2&2\\ &&&2 \end{matrix}$\\
\oddrow $\begin{array}{@{}l@{}}\oddrow X_{3^2, 4^3}\\\quad\subset\PP(1^6, 2)\end{array}$&&$14$&$-5$&$120$&$\frac{1}{2}(1, 1, 1, 3, 3)$&$\footnotesize\begin{matrix} 1&1&2&2\\ &1&2&2\\ &&2&2\\ &&&3 \end{matrix}$\\
\evnrow $\begin{array}{@{}l@{}}\evnrow X_{3, 4^3, 5}\\\quad\subset\PP(1^5, 2^2)\end{array}$&$\frac{1}{2}(1, 1, 1)$&$\frac{19}{2}$&$-4$&$\frac{195}{2}$&$(0, 1, 1, 1, 2)$&$\footnotesize\begin{matrix} 1&1&1&2\\ &2&2&3\\ &&2&3\\ &&&3 \end{matrix}$\\
\oddrow $\begin{array}{@{}l@{}}\oddrow X_{4^5}\\\quad\subset\PP(1^5, 2^2)\end{array}$&&$10$&$-4$&$96$&$(1, 1, 1, 1, 1)$&$\footnotesize\begin{matrix} 2&2&2&2\\ &2&2&2\\ &&2&2\\ &&&2 \end{matrix}$\\
\evnrow $\begin{array}{@{}l@{}}\evnrow X_{4^3, 5^2}\\\quad\subset\PP(1^4, 2^3)\end{array}$&$3\times \frac{1}{2}(1, 1, 1)$&$\frac{13}{2}$&$-3$&$\frac{153}{2}$&$\frac{1}{2}(1, 1, 3, 3, 3)$&$\footnotesize\begin{matrix} 1&2&2&2\\ &2&2&2\\ &&3&3\\ &&&3 \end{matrix}$\\
\oddrow $\begin{array}{@{}l@{}}\oddrow X_{4^2, 5^2, 6}\\\quad\subset\PP(1^4, 2^2, 3)\end{array}$&$\frac{1}{2}(1, 1, 1)$&$\frac{11}{2}$&$-3$&$\frac{147}{2}$&$(0, 1, 1, 2, 2)$&$\footnotesize\begin{matrix} 1&1&2&2\\ &2&3&3\\ &&3&3\\ &&&4 \end{matrix}$\\
\evnrow $\begin{array}{@{}l@{}}\evnrow X_{4, 5^2, 6^2}\\\quad\subset\PP(1^3, 2^3, 3)\end{array}$&$5\times \frac{1}{2}(1, 1, 1)$&$\frac{7}{2}$&$-2$&$\frac{111}{2}$&$\frac{1}{2}(1, 1, 3, 3, 5)$&$\footnotesize\begin{matrix} 1&2&2&3\\ &2&2&3\\ &&3&4\\ &&&4 \end{matrix}$\\
\oddrow $\begin{array}{@{}l@{}}\oddrow X_{4, 5, 6^2, 7}\\\quad\subset\PP(1^3, 2^2, 3^2)\end{array}$&$\frac{1}{2}(1, 1, 1), \frac{1}{3}(1, 2, 2)$&$\frac{17}{6}$&$-2$&$\frac{313}{6}$&$(0, 1, 1, 2, 3)$&$\footnotesize\begin{matrix} 1&1&2&3\\ &2&3&4\\ &&3&4\\ &&&5 \end{matrix}$\\
\evnrow $\begin{array}{@{}l@{}}\evnrow X_{5^2, 6^3}\\\quad\subset\PP(1^3, 2^2, 3^2)\end{array}$&$2\times \frac{1}{2}(1, 1, 1)$&$3$&$-2$&$51$&$(1, 1, 1, 2, 2)$&$\footnotesize\begin{matrix} 2&2&3&3\\ &2&3&3\\ &&3&3\\ &&&4 \end{matrix}$\\
\oddrow $\begin{array}{@{}l@{}}\oddrow X_{5, 6^3, 7}\\\quad\subset\PP(1^2, 2^3, 3^2)\end{array}$&$7\times \frac{1}{2}(1, 1, 1), \frac{1}{3}(1, 2, 2)$&$\frac{11}{6}$&$-1$&$\frac{223}{6}$&$\frac{1}{2}(1, 3, 3, 3, 5)$&$\footnotesize\begin{matrix} 2&2&2&3\\ &3&3&4\\ &&3&4\\ &&&4 \end{matrix}$\\
\evnrow $\begin{array}{@{}l@{}}\evnrow X_{6^3, 7^2}\\\quad\subset\PP(1^2, 2^2, 3^3)\end{array}$&$\frac{1}{2}(1, 1, 1), 3\times \frac{1}{3}(1, 2, 2)$&$\frac{3}{2}$&$-1$&$\frac{67}{2}$&$(1, 1, 2, 2, 2)$&$\footnotesize\begin{matrix} 2&3&3&3\\ &3&3&3\\ &&4&4\\ &&&4 \end{matrix}$\\
\oddrow $\begin{array}{@{}l@{}}\oddrow X_{6^2, 7^2, 8}\\\quad\subset\PP(1^2, 2^2, 3^2, 4)\end{array}$&$4\times \frac{1}{2}(1, 1, 1), \frac{1}{3}(1, 2, 2)$&$\frac{4}{3}$&$-1$&$\frac{98}{3}$&$\frac{1}{2}(1, 3, 3, 5, 5)$&$\footnotesize\begin{matrix} 2&2&3&3\\ &3&4&4\\ &&4&4\\ &&&5 \end{matrix}$\\
\evnrow $\begin{array}{@{}l@{}}\evnrow X_{6, 7, 8, 9, 10}\\\quad\subset\PP(1^2, 2, 3^2, 4, 5)\end{array}$&$\frac{1}{2}(1, 1, 1), \frac{1}{3}(1, 2, 2)$&$\frac{5}{6}$&$-1$&$\frac{169}{6}$&$(0, 1, 2, 3, 4)$&$\footnotesize\begin{matrix} 1&2&3&4\\ &3&4&5\\ &&5&6\\ &&&7 \end{matrix}$\\
\oddrow $\begin{array}{@{}l@{}}\oddrow X_{7, 8^2, 9, 10}\\\quad\subset\PP(1, 2^2, 3^2, 4, 5)\end{array}$&$7\times \frac{1}{2}(1, 1, 1), 3\times \frac{1}{3}(1, 2, 2)$&$\frac{1}{2}$&$0$&$\frac{37}{2}$&$\frac{1}{2}(1, 3, 5, 5, 7)$&$\footnotesize\begin{matrix} 2&3&3&4\\ &4&4&5\\ &&5&6\\ &&&6 \end{matrix}$\\
\evnrow $\begin{array}{@{}l@{}}\evnrow X_{8, 9^2, 10^2}\\\quad\subset\PP(1, 2, 3^2, 4^2, 5)\end{array}$&$3\times \frac{1}{2}(1, 1, 1), \frac{1}{3}(1, 2, 2), 2\times \frac{1}{4}(1, 3, 3)$&$\frac{1}{3}$&$0$&$\frac{44}{3}$&$\frac{1}{2}(3, 3, 5, 5, 7)$&$\footnotesize\begin{matrix} 3&4&4&5\\ &4&4&5\\ &&5&6\\ &&&6 \end{matrix}$\\
\oddrow $\begin{array}{@{}l@{}}\oddrow X_{8, 9, 10^2, 11}\\\quad\subset\PP(1, 2, 3^2, 4, 5^2)\end{array}$&$\frac{1}{2}(1, 1, 1), 3\times \frac{1}{3}(1, 2, 2), \frac{1}{5}(2, 3, 4)$&$\frac{3}{10}$&$0$&$\frac{143}{10}$&$(1, 2, 2, 3, 4)$&$\footnotesize\begin{matrix} 3&3&4&5\\ &4&5&6\\ &&5&6\\ &&&7 \end{matrix}$\\
\evnrow $\begin{array}{@{}l@{}}\evnrow X_{12, 13, 14, 15, 16}\\\quad\subset\PP(1, 3, 4, 5, 6, 7, 8)\end{array}$&$\frac{1}{2}(1, 1, 1), \frac{1}{3}(1, 2, 2), \frac{1}{4}(1, 3, 3)$&$\frac{1}{12}$&$0$&$\frac{95}{12}$&$\frac{1}{2}(3, 5, 7, 9, 11)$&$\footnotesize\begin{matrix} 4&5&6&7\\ &6&7&8\\ &&8&9\\ &&&10 \end{matrix}$\\
\oddrow $\begin{array}{@{}l@{}}\oddrow X_{12, 13, 14, 15, 16}\\\quad\subset\PP(3, 4^2, 5^2, 6, 7)\end{array}$&$\begin{array}{@{}l@{}}2\times \frac{1}{2}(1, 1, 1), \frac{1}{3}(1, 2, 2), 2\times \frac{1}{4}(1, 3, 3),\\\quad 2\times \frac{1}{5}(1, 4, 4), \frac{1}{5}(2, 3, 4)\end{array}$&$\frac{1}{30}$&$1$&$\frac{107}{30}$&$\frac{1}{2}(3, 5, 7, 9, 11)$&$\footnotesize\begin{matrix} 4&5&6&7\\ &6&7&8\\ &&8&9\\ &&&10 \end{matrix}$\\
\end{longtable}
\begin{longtable}{>{\hspace{0.5em}}llcccccr<{\hspace{0.5em}}}
\caption{Codimension five.} \label{tab!codim5}\\
\toprule
\multicolumn{1}{c}{Variety}&\multicolumn{1}{c}{Basket $\mathcal{B}$}&\multicolumn{1}{c}{$K^3_X$}&\multicolumn{1}{c}{$\chi$}&\multicolumn{1}{c}{$K_Xc_2$}&\multicolumn{1}{c}{$u$ and $w$}&\multicolumn{2}{c}{Variable weights $x,x_i,x_{ij}$}\\
\cmidrule(lr){1-1}\cmidrule(lr){2-2}\cmidrule(lr){3-5}\cmidrule(lr){6-6}\cmidrule(lr){7-8}
\endfirsthead
\multicolumn{8}{l}{\vspace{-0.25em}\scriptsize\emph{\tablename\ \thetable{} continued from previous page}}\\
\midrule
\endhead
\multicolumn{8}{r}{\scriptsize\emph{Continued on next page}}\\
\endfoot
\bottomrule
\endlastfoot
\evnrow $\begin{array}{@{}l@{}}\evnrow X_{2, 3^8, 4}\\\quad\subset\PP(1^7, 2^2)\end{array}$&$2\times \frac{1}{2}(1, 1, 1)$&$21$&$-6$&$147$&$\begin{array}{@{}c@{}}1\\(0, 0, 0, 0, 1)\end{array}$&$\begin{array}{@{}c@{}}1\\ 2, 2, 2, 2, 1 \end{array}$&$\footnotesize\begin{matrix} 1&1&1&2\\ &1&1&2\\ &&1&2\\ &&&2 \end{matrix}$\\
\oddrow $\begin{array}{@{}l@{}}\oddrow X_{3^5, 4^5}\\\quad\subset\PP(1^5, 2^4)\end{array}$&$5\times \frac{1}{2}(1, 1, 1)$&$\frac{23}{2}$&$-4$&$\frac{207}{2}$&$\begin{array}{@{}c@{}}1\\\frac{1}{2}(-1, 1, 1, 1, 1)\end{array}$&$\begin{array}{@{}c@{}}1\\ 3, 2, 2, 2, 2 \end{array}$&$\footnotesize\begin{matrix} 1&1&1&1\\ &2&2&2\\ &&2&2\\ &&&2 \end{matrix}$\\
\evnrow $\begin{array}{@{}l@{}}\evnrow X_{3^5, 4^5}\\\quad\subset\PP(1^6, 2^2, 3)\end{array}$&$\frac{1}{3}(1, 2, 2)$&$\frac{46}{3}$&$-5$&$\frac{368}{3}$&$\begin{array}{@{}c@{}}1\\\frac{1}{2}(-1, 1, 1, 1, 1)\end{array}$&$\begin{array}{@{}c@{}}1\\ 3, 2, 2, 2, 2 \end{array}$&$\footnotesize\begin{matrix} 1&1&1&1\\ &2&2&2\\ &&2&2\\ &&&2 \end{matrix}$\\
\oddrow $\begin{array}{@{}l@{}}\oddrow X_{3^2, 4^6, 5^2}\\\quad\subset\PP(1^4, 2^4, 3)\end{array}$&$4\times \frac{1}{2}(1, 1, 1), \frac{1}{3}(1, 2, 2)$&$\frac{22}{3}$&$-3$&$\frac{242}{3}$&$\begin{array}{@{}c@{}}1\\(0, 0, 0, 1, 1)\end{array}$&$\begin{array}{@{}c@{}}1\\ 3, 3, 3, 2, 2 \end{array}$&$\footnotesize\begin{matrix} 1&1&2&2\\ &1&2&2\\ &&2&2\\ &&&3 \end{matrix}$\\
\evnrow $\begin{array}{@{}l@{}}\evnrow X_{4^{10}}\\\quad\subset\PP(1^3, 2^6)\end{array}$&$12\times \frac{1}{2}(1, 1, 1)$&$6$&$-2$&$66$&$\begin{array}{@{}c@{}}2\\(0, 0, 0, 0, 0)\end{array}$&$\begin{array}{@{}c@{}}2\\ 2, 2, 2, 2, 2 \end{array}$&$\footnotesize\begin{matrix} 2&2&2&2\\ &2&2&2\\ &&2&2\\ &&&2 \end{matrix}$\\
\oddrow $\begin{array}{@{}l@{}}\oddrow X_{4^3, 5^4, 6^3}\\\quad\subset\PP(1^3, 2^3, 3^2, 4)\end{array}$&$3\times \frac{1}{2}(1, 1, 1), \frac{1}{4}(1, 3, 3)$&$\frac{15}{4}$&$-2$&$\frac{225}{4}$&$\begin{array}{@{}c@{}}1\\(0, 0, 1, 1, 1)\end{array}$&$\begin{array}{@{}c@{}}1\\ 4, 4, 3, 3, 3 \end{array}$&$\footnotesize\begin{matrix} 1&2&2&2\\ &2&2&2\\ &&3&3\\ &&&3 \end{matrix}$\\
\evnrow $\begin{array}{@{}l@{}}\evnrow X_{4^3, 5^4, 6^3}\\\quad\subset\PP(1^2, 2^4, 3^3)\end{array}$&$7\times \frac{1}{2}(1, 1, 1), 3\times \frac{1}{3}(1, 2, 2)$&$\frac{5}{2}$&$-1$&$\frac{85}{2}$&$\begin{array}{@{}c@{}}1\\(0, 0, 1, 1, 1)\end{array}$&$\begin{array}{@{}c@{}}1\\ 4, 4, 3, 3, 3 \end{array}$&$\footnotesize\begin{matrix} 1&2&2&2\\ &2&2&2\\ &&3&3\\ &&&3 \end{matrix}$\\
\oddrow $\begin{array}{@{}l@{}}\oddrow X_{4, 5^2, 6^4, 7^2, 8}\\\quad\subset\PP(1^2, 2^3, 3^2, 4, 5)\end{array}$&$6\times \frac{1}{2}(1, 1, 1), \frac{1}{5}(2, 3, 4)$&$\frac{9}{5}$&$-1$&$\frac{189}{5}$&$\begin{array}{@{}c@{}}1\\(0, 0, 1, 1, 2)\end{array}$&$\begin{array}{@{}c@{}}1\\ 5, 5, 4, 4, 3 \end{array}$&$\footnotesize\begin{matrix} 1&2&2&3\\ &2&2&3\\ &&3&4\\ &&&4 \end{matrix}$\\
\evnrow $\begin{array}{@{}l@{}}\evnrow X_{4, 5^2, 6^4, 7^2, 8}\\\quad\subset\PP(1^2, 2^2, 3^3, 4^2)\end{array}$&$2\times \frac{1}{2}(1, 1, 1), 2\times \frac{1}{4}(1, 3, 3)$&$\frac{3}{2}$&$-1$&$\frac{69}{2}$&$\begin{array}{@{}c@{}}1\\(0, 0, 1, 1, 2)\end{array}$&$\begin{array}{@{}c@{}}1\\ 5, 5, 4, 4, 3 \end{array}$&$\footnotesize\begin{matrix} 1&2&2&3\\ &2&2&3\\ &&3&4\\ &&&4 \end{matrix}$\\
\oddrow $\begin{array}{@{}l@{}}\oddrow X_{4, 5^2, 6^4, 7^2, 8}\\\quad\subset\PP(1, 2^3, 3^4, 4)\end{array}$&$6\times \frac{1}{2}(1, 1, 1), 6\times \frac{1}{3}(1, 2, 2)$&$1$&$0$&$25$&$\begin{array}{@{}c@{}}1\\(0, 0, 1, 1, 2)\end{array}$&$\begin{array}{@{}c@{}}1\\ 5, 5, 4, 4, 3 \end{array}$&$\footnotesize\begin{matrix} 1&2&2&3\\ &2&2&3\\ &&3&4\\ &&&4 \end{matrix}$\\
\evnrow $\begin{array}{@{}l@{}}\evnrow X_{5^2, 6^6, 7^2}\\\quad\subset\PP(1, 2^3, 3^4, 4)\end{array}$&$\begin{array}{@{}l@{}}7\times \frac{1}{2}(1, 1, 1), 4\times \frac{1}{3}(1, 2, 2),\\\quad \frac{1}{4}(1, 3, 3)\end{array}$&$\frac{13}{12}$&$0$&$\frac{299}{12}$&$\begin{array}{@{}c@{}}2\\(0, 0, 0, 1, 1)\end{array}$&$\begin{array}{@{}c@{}}2\\ 4, 4, 4, 3, 3 \end{array}$&$\footnotesize\begin{matrix} 2&2&3&3\\ &2&3&3\\ &&3&3\\ &&&4 \end{matrix}$\\
\oddrow $\begin{array}{@{}l@{}}\oddrow X_{6^3, 7^4, 8^3}\\\quad\subset\PP(1, 2^2, 3^3, 4^2, 5)\end{array}$&$\begin{array}{@{}l@{}}5\times \frac{1}{2}(1, 1, 1), 3\times \frac{1}{3}(1, 2, 2),\\\quad \frac{1}{5}(1, 4, 4)\end{array}$&$\frac{7}{10}$&$0$&$\frac{203}{10}$&$\begin{array}{@{}c@{}}2\\(0, 0, 1, 1, 1)\end{array}$&$\begin{array}{@{}c@{}}2\\ 5, 5, 4, 4, 4 \end{array}$&$\footnotesize\begin{matrix} 2&3&3&3\\ &3&3&3\\ &&4&4\\ &&&4 \end{matrix}$\\
\evnrow $\begin{array}{@{}l@{}}\evnrow X_{6, 7^2, 8^4, 9^2, 10}\\\quad\subset\PP(2^2, 3^3, 4^2, 5^2)\end{array}$&$\begin{array}{@{}l@{}}8\times \frac{1}{2}(1, 1, 1), 5\times \frac{1}{3}(1, 2, 2), \\\quad 2\times \frac{1}{5}(2, 3, 4)\end{array}$&$\frac{4}{15}$&$1$&$\frac{164}{15}$&$\begin{array}{@{}c@{}}2\\(0, 0, 1, 1, 2)\end{array}$&$\begin{array}{@{}c@{}}2\\ 6, 6, 5, 5, 4 \end{array}$&$\footnotesize\begin{matrix} 2&3&3&4\\ &3&3&4\\ &&4&5\\ &&&5 \end{matrix}$\\
\oddrow $\begin{array}{@{}l@{}}\oddrow X_{6, 7^2, 8^4, 9^2, 10}\\\quad\subset\PP(1, 2, 3^2, 4^3, 5^2)\end{array}$&$4\times \frac{1}{2}(1, 1, 1), 2\times \frac{1}{5}(1, 4, 4)$&$\frac{2}{5}$&$0$&$\frac{78}{5}$&$\begin{array}{@{}c@{}}2\\(0, 0, 1, 1, 2)\end{array}$&$\begin{array}{@{}c@{}}2\\ 6, 6, 5, 5, 4 \end{array}$&$\footnotesize\begin{matrix} 2&3&3&4\\ &3&3&4\\ &&4&5\\ &&&5 \end{matrix}$\\
\evnrow $\begin{array}{@{}l@{}}\evnrow X_{6, 7, 8^2, 9^2, 10^2, 11, 12}\\\quad\subset\PP(1, 2, 3^2, 4^2, 5, 6, 7)\end{array}$&$\begin{array}{@{}l@{}}3\times \frac{1}{2}(1, 1, 1), \frac{1}{4}(1, 3, 3),\\\quad \frac{1}{7}(3, 4, 6)\end{array}$&$\frac{9}{28}$&$0$&$\frac{423}{28}$&$\begin{array}{@{}c@{}}1\\(0, 1, 1, 2, 3)\end{array}$&$\begin{array}{@{}c@{}}1\\ 8, 7, 7, 6, 5 \end{array}$&$\footnotesize\begin{matrix} 2&2&3&4\\ &3&4&5\\ &&4&5\\ &&&6 \end{matrix}$\\
\oddrow $\begin{array}{@{}l@{}}\oddrow X_{8^3, 9^4, 10^3}\\\quad\subset\PP(2, 3^2, 4^3, 5^3)\end{array}$&$\begin{array}{@{}l@{}}4\times \frac{1}{2}(1, 1, 1), 3\times \frac{1}{4}(1, 3, 3), \\\quad 3\times \frac{1}{5}(2, 3, 4)\end{array}$&$\frac{3}{20}$&$1$&$\frac{153}{20}$&$\begin{array}{@{}c@{}}3\\(0, 0, 1, 1, 1)\end{array}$&$\begin{array}{@{}c@{}}3\\ 6, 6, 5, 5, 5 \end{array}$&$\footnotesize\begin{matrix} 3&4&4&4\\ &4&4&4\\ &&5&5\\ &&&5 \end{matrix}$\\
\evnrow $\begin{array}{@{}l@{}}\evnrow X_{8, 9^2, 10^4, 11^2, 12}\\\quad\subset\PP(2, 3^2, 4^2, 5^2, 6, 7)\end{array}$&$\begin{array}{@{}l@{}}4\times \frac{1}{2}(1, 1, 1), 4\times \frac{1}{3}(1, 2, 2),\\\quad 2\times \frac{1}{4}(1, 3, 3), \frac{1}{7}(2, 5, 6)\end{array}$&$\frac{5}{42}$&$1$&$\frac{295}{42}$&$\begin{array}{@{}c@{}}3\\(0, 0, 1, 1, 2)\end{array}$&$\begin{array}{@{}c@{}}3\\ 7, 7, 6, 6, 5 \end{array}$&$\footnotesize\begin{matrix} 3&4&4&5\\ &4&4&5\\ &&5&6\\ &&&6 \end{matrix}$\\
\oddrow $\begin{array}{@{}l@{}}\oddrow X_{8, 9, 10^2, 11^2, 12^2, 13, 14}\\\quad\subset\PP(2, 3^2, 4, 5^2, 6, 7, 8)\end{array}$&$\begin{array}{@{}l@{}}3\times \frac{1}{2}(1, 1, 1), 5\times \frac{1}{3}(1, 2, 2),\\\quad \frac{1}{5}(2, 3, 4), \frac{1}{8}(3, 5, 7)\end{array}$&$\frac{11}{120}$&$1$&$\frac{781}{120}$&$\begin{array}{@{}c@{}}2\\(0, 1, 1, 2, 3)\end{array}$&$\begin{array}{@{}c@{}}2\\ 9, 8, 8, 7, 6 \end{array}$&$\footnotesize\begin{matrix} 3&3&4&5\\ &4&5&6\\ &&5&6\\ &&&7 \end{matrix}$\\
\evnrow $\begin{array}{@{}l@{}}\evnrow X_{10, 11^2, 12^4, 13^2, 14}\\\quad\subset\PP(3, 4^2, 5^2, 6^2, 7^2)\end{array}$&$\begin{array}{@{}l@{}}3\times \frac{1}{2}(1, 1, 1), 2\times \frac{1}{5}(1, 4, 4),\\\quad 2\times \frac{1}{7}(3, 4, 6)\end{array}$&$\frac{3}{70}$&$1$&$\frac{267}{70}$&$\begin{array}{@{}c@{}}4\\(0, 0, 1, 1, 2)\end{array}$&$\begin{array}{@{}c@{}}4\\ 8, 8, 7, 7, 6 \end{array}$&$\footnotesize\begin{matrix} 4&5&5&6\\ &5&5&6\\ &&6&7\\ &&&7 \end{matrix}$\\
\oddrow $\begin{array}{@{}l@{}}\oddrow X_{12, 13, 14^2, 15^2, 16^2, 17, 18}\\\quad\subset\PP(3, 4, 5, 6, 7^2, 8, 9, 10)\end{array}$&$\begin{array}{@{}l@{}}\frac{1}{2}(1, 1, 1), \frac{1}{4}(1, 3, 3), \frac{1}{5}(2, 3, 4),\\\quad \frac{1}{7}(3, 4, 6), \frac{1}{10}(3, 7, 9)\end{array}$&$\frac{3}{140}$&$1$&$\frac{393}{140}$&$\begin{array}{@{}c@{}}4\\(0, 1, 1, 2, 3)\end{array}$&$\begin{array}{@{}c@{}}4\\ 11, 10, 10, 9, 8 \end{array}$&$\footnotesize\begin{matrix} 5&5&6&7\\ &6&7&8\\ &&7&8\\ &&&9 \end{matrix}$\\
\evnrow $\begin{array}{@{}l@{}}\evnrow X_{12, 13, 14, 15, 16^2, 17, 18, 19, 20}\\\quad\subset\PP(3, 4, 5, 6, 7, 8, 9, 10, 11)\end{array}$&$\begin{array}{@{}l@{}}2\times \frac{1}{2}(1, 1, 1), 3\times \frac{1}{3}(1, 2, 2),\\\quad \frac{1}{5}(1, 4, 4), \frac{1}{11}(4, 7, 10)\end{array}$&$\frac{1}{55}$&$1$&$\frac{149}{55}$&$\begin{array}{@{}c@{}}3\\(0, 1, 2, 3, 4)\end{array}$&$\begin{array}{@{}c@{}}3\\ 13, 12, 11, 10, 9 \end{array}$&$\footnotesize\begin{matrix} 4&5&6&7\\ &6&7&8\\ &&8&9\\ &&&10 \end{matrix}$\\
\end{longtable}
\end{landscape}
\bibliographystyle{amsalpha}
\providecommand{\bysame}{\leavevmode\hbox to3em{\hrulefill}\thinspace}
\providecommand{\MR}{\relax\ifhmode\unskip\space\fi MR }
\providecommand{\MRhref}[2]{%
  \href{http://www.ams.org/mathscinet-getitem?mr=#1}{#2}
}
\providecommand{\href}[2]{#2}

\end{document}